\documentclass{amsart}

\usepackage{amsmath,amsthm,amssymb,url}

\newtheorem{theorem}{Theorem}[section]
\newtheorem{lemma}[theorem]{Lemma}

\newcommand{\ph}{\varphi}
\newcommand{\NN}{\mathbb{N}}

\newcommand{\ZZ}{\mathbb{Z}}

\newcommand{\st}{\; | \;}

\newcommand{\mdl}[1]{{\mathcal #1}}

\title{Ultraproducts and metastability}

\author{Jeremy Avigad}
\address{Department of Philosophy and Department of Mathematical
  Sciences\\
Carnegie Mellon University\\
Pittsburgh, Pennsylvania 15213}
\email{avigad@cmu.edu}

\author{Jos\'e Iovino}
\address{Department of Mathematics\\
The University of Texas at San Antonio \\
  San Antonio, Texas 78249
and Department of Mathematical Sciences\\
Carnegie Mellon University\\
Pittsburgh, Pennsylvania 15213}
\email{iovino@math.utsa.edu}



\thanks{Avigad's work has been partially supported by NSF grant DMS-1068829 and AFOSR grant FA9550-12-1-0370. We are grateful to Ulrich Kohlenbach and Terence Tao for helpful suggestions and corrections.}

\begin{document}

\begin{abstract}
Given a convergence theorem in analysis, under very general conditions a model-theoretic compactness argument implies that there is a uniform bound on the rate of metastability. We illustrate with three examples from ergodic theory.
\end{abstract}

\maketitle

\section{Introduction}
\label{introduction:section}

Convergence theorems in analysis are often disappointingly nonuniform. For example, Krengel \cite{krengel:78} has shown, roughly speaking, that even if one fixes an ergodic measure preserving system, the convergence of averages guaranteed by the mean ergodic theorem can be arbitrarily slow. Our goal here is to show that even in such cases, a compactness argument can often be used to establish a weaker uniformity, namely, the existence of uniform bounds on the rate of metastable convergence. 

If $(a_n)_{n \in \NN}$ is a sequence of elements in a metric space $(X, d)$, saying that $(a_n)$ is Cauchy is equivalent to saying that, for every $\varepsilon > 0$ and function $F : \NN \to \NN$, there is an $n$ such that $d(a_i,a_j) < \varepsilon$ for every $i, j \in [n, F(n)]$. Think of $F$ as trying to disprove the convergence of $(a_n)$ by finding intervals where the sequence fluctuates by more than $\varepsilon$; the $n$ asserted to exist foils $F$ in the sense that the sequence remains $\varepsilon$-stable on $[n, F(n)]$. We will call a bound on such an $n$, depending on $F$ and $\varepsilon$, a \emph{bound on the rate of metastability}.

The arguments below show that, in many convergence theorems, there is a bound on the rate of metastability that depends on only a few of the relevant parameters. All is that required is that the class of structures in question, and the hypotheses of the theorem, are preserved under a certain model-theoretic ultraproduct construction in which these parameters remain fixed. A sufficient condition for this can be formulated in syntactic terms, by asserting that the the relevant hypotheses and axioms can be put in a certain logical form. Section~\ref{background:section} summarizes the necessary background on ultraproducts in analysis, and presents a theorem which characterizes the existence of a uniform bound on the rate of metastability of a collection of sequences in terms of the convergence of ultraproducts of those sequences. Section~\ref{examples:section} illustrates the use of this equivalence with three examples from ergodic theory.

Metastability has proved useful in ergodic theory and ergodic Ramsey theory \cite{green:tao:08,tao:06,tao:08}; see also \cite[Sections 1.3--1.4]{tao:08b}, and \cite{avigad:et:al:10,avigad:dean:rute:12,kohlenbach:leustean:09,kohlenbach:10,kohlenbach:12,kohlenbach:leustean:12,kohlenbach:schade:12} for various instances of metastability in analysis. Sometimes stronger uniformities are available than the ones we consider here, in the form of variational inequalities (e.g.~\cite{jones:et:al:96,jones:et:al:98,jones:et:al:03,avigad:rute:unp}). Bergelson et al.~\cite{bergelson:et:al:00} explore aspects of uniformity in ergodic theory and ergodic Ramsey theory, but most of the methods there rely on specific combinatorial features of the phenomena under consideration.

The methods developed here complement proof-theoretic methods developed by Kohlenbach and collaborators, e.g.~in \cite{kohlenbach:05,gerhardy:kohlenbach:08}. Roughly, those methods provide ``metatheorems'' which show that when a statement with a certain logical form is derivable in a certain (fairly expressive) axiomatic theory, certain uniformities always obtain. The arguments we present here replace derivability in an axiomatic system with closure under the formation of ultraproducts. Indeed, it seems likely that such arguments can be used to establish general metatheorems likes the ones in \cite{kohlenbach:05,gerhardy:kohlenbach:08}, by considering ultraproducts of models of the axiomatic theories in question. 

It is worth noting that although the methods we describe here can be used to establish the existence of a very uniform bound, they give no explicit quantitative information at all, nor even show that it is possible to \emph{compute} such a bound as a function of $F$ and $\varepsilon$. In contrast, the proof-theoretic techniques provide ways that such information can be ``mined'' from a specific proof. In particular, the general metatheorems described in the last paragraph guarantee that the associated bounds are computable. If one is primarily interested in uniformity, however, the methods here have the virtue of being easy to understand and apply.

\section{Ultraproducts of Banach spaces}
\label{background:section}

In this section we review standard ultraproduct constructions in analysis; see \cite{ben:yaacov:et:al:08,heinrich:80,henson:iovino:02,loeb:wolf:00} for more details.

Let $I$ be any infinite set, and let $D$ be a nonprincipal ultrafilter on $I$. (Below, we will always take $I$ to be $\NN$.) Any bounded sequence $(r_i)_{i \in I}$ of real numbers has a unique limit $r$ with respect to $D$, written $r = \lim_{i,D} r_i$; this means that for every $\varepsilon > 0$ the set $\{ i \in I \st | r_i - r | < \varepsilon \}$ is in $D$. Suppose that for each $i$, $(X_i, d_i)$ is a metric space with a distinguished point $a_i$. Let
\[
 X_\infty = \Big\{ (x_i) \in \prod_{i \in I} X_i \; \big| \; \sup_i d(x_i, a_i) < \infty \Big\} \mathop{\mbox{\larger\larger{$/$}}} \sim,
\]
where $(x_i) \sim (y_i)$ if and only if $\lim_{i,D} d(x_i, y_i) = 0$. Let $d_\infty$ be the metric on $X_\infty$ defined by $d_\infty((x_i), (y_i)) = \lim_{i,D} d(x_i, y_i)$. Leaving the dependence on the choice of the base points $a_i$ implicit, we will call this an \emph{ultraproduct} of the metric spaces $(X_i, d_i)$, denoted by $\left( \prod_{i \in I} (X_i, d_i) \right)_D$. If there is a uniform bound on the diameters of these spaces, the choice of the sequence $(a_i)$ of ``anchor points'' is clearly irrelevant.

This ultraproduct construction is an instance of Luxemburg's nonstandard hull construction \cite{luxemburg:69}. We can extend it to ultraproducts of a sequence $(X_i)$ of normed spaces using $a_i = 0$ and the distance given by the norm. Ultraproducts of Banach spaces were introduced by Dacunha-Castelle and Krivine \cite{dacunha:castelle:krivine:72}, and are an important tool in a number of branches of analysis (see e.g.~\cite{henson:iovino:02}).

In first-order model theory, one can take an ultraproduct of any sequence of structures $\mdl M_i$, and {\L}os's theorem says that any first-order sentence $\ph$ is true in the ultraproduct if and only if it is true in almost every $\mdl M_i$, in the sense of $D$; in other words, if and only if $\{ i \st \mdl M_i \models \ph \} \in D$. The constructions above, however, are not ultraproducts in the first-order sense, since we restrict to ``finite'' elements, mod out by infinitesimal proximity $\sim$, and (implicitly, by taking limits with respect to $D$) pass to the standard part of nonstandard distances and norms. This gives rise to two complications.

First, if we extend the metric or normed spaces with other functions, their lifting to the ultraproduct will not be well defined if they fail to map finite elements to finite elements, or fail to respect $\sim$. We can lift, however, any family $(f_i)$ of functions that satisfies an appropriate \emph{uniform boundedness condition} (roughly, elements of the family are uniformly bounded on bounded sets around the base point) and an appropriate \emph{uniform continuity condition} (which is to say that there is a uniform modulus of uniform continuity on such sets). The resulting function on the ultraproduct will be denoted $\left(\prod_i f_i\right)_D$. For details, see \cite[Section 4]{henson:iovino:02} or \cite[Section 4]{ben:yaacov:et:al:08}.

Second, {\L}os's theorem needs to be modified. One strategy, described in \cite{henson:iovino:02}, is to restrict attention to a class of \emph{positively bounded formulas}. These are formulas generated from atomic formulas $r \leq t$ and $t \leq r$, where $t$ is an appropriate term and $r$ is rational, using only the positive connectives $\land$ and $\lor$, as well as universal and existential quantification over compact balls in the structure. An \emph{approximation} to such a formula is obtained by replacing each $r$ in an atomic formula $r \leq t$ by any $r' < r$, and each $r$ in an atomic formula $t \leq r$ by any $r' > r$. Say that a formula $\ph$ with parameters is \emph{approximately true} in a structure if every approximation $\ph'$ to $\ph$ is true in the structure. One can then show that if $a_1, \ldots, a_n$ are elements of the ultraproduct with each $a_j$ represented by the sequence $(a_{j, i})_{i \in I}$, then a positively bounded formula $\ph(a_1, \ldots, a_n)$ is approximately true in the ultraproduct $\left(\prod_{i \in I} \mdl M_i\right)_D$ if and only if 
\[
   \{ i \in I \st \mdl M_i \models \ph'(a_{1,i}, \ldots, a_{n,i}) \} \in D
\]
for every approximation $\ph'$ to $\ph$.

Suppose $\Gamma$ is a set of positively bounded sentences, and $C$ is the class of structures that approximately satisfy each sentence in $\Gamma$. The previous equivalence implies that $C$ is closed under ultraproducts. In fact, Henson and Iovino \cite[Proposition 13.6]{henson:iovino:02} show that a class of structures $C$ can be axiomatized in this way if and only if $C$ is closed under isomorphisms, ultraproducts, and ultraroots. 

Another strategy, described in \cite{ben:yaacov:et:al:08}, is to modify first-order semantics so that formulas take on truth values in a bounded interval of reals, in which case the truth value of a formula $\ph$ in the ultraproduct is the $D$-limit of its truth values in the individual structures. Spelling out the details here would take us too far afield. Below we will only use the fact that certain classes of structures and hypotheses are preserved under ultraproducts, as well as the easy fact that a quantifier-free positively bounded formula $\ph$ is true in a structure if and only if every approximation to it is true, thereby simplifying the equivalence above.

If the ultrafilter $D$ is nonprincipal, an ultraproduct $(\prod_{i \in I} \mdl M_i)_D$ of metric spaces or normed spaces is \emph{$\aleph_1$-saturated}, or \emph{countably saturated}, in the following sense: if $\Gamma$ is a countable set of positive bounded formulas of the form $\ph(x_1,\dots, x_n)$ such that every finite set of approximations of formulas in $\Gamma$ is satisfied by some $n$-tuple  of elements of $(\prod_i \mdl M_i)_D$, then there exists an $n$-tuple  of elements of $(\prod_i \mdl M_i)_D$ that satisfies all the formulas in $\Gamma$ \cite[Proposition 9.18]{henson:iovino:02}.  In particular, every ultraproduct over a nonprincipal ultrafilter is metrically complete \cite[Proposition 9.21]{henson:iovino:02}.  (For an arbitrary infinite cardinal $\kappa$, the concept of $\kappa$-saturation is defined similarly, by replacing $\aleph_1$ above with $\kappa$. If the cardinality of $I$ is sufficiently large relative to $\kappa$, the structure  $(\prod_{i \in I} \mdl M_i)_D$ can be made be $\kappa$-saturated with a careful choice of $D$ \cite[Theorem 10.8]{henson:iovino:02}.)

The following theorem provides a neat characterization of the relationship between convergence in ultraproducts and uniformity.
 
\begin{theorem}
\label{main:theorem}
Let $C$ be any collection of pairs $((X, d), (a_n)_{n \in \NN})$, where each $(a_n)$ is a sequence of elements in the corresponding metric space $(X, d)$. For any nonprincipal ultrafilter on $D$, the following statements are equivalent:
\begin{enumerate}
\item There is a uniform bound on the rate of metastability for the sequences $(a_n)$. In other words, for every $F : \NN \to \NN$ and $\varepsilon > 0$, there is a $b$ with the following property: for every pair $((X, d), (a_n)_{n \in \NN})$ in $C$, there is an $n \leq b$ such that $d(a_i, a_j) < \varepsilon$ for every $i, j \in [n, F(n)]$.
\item For any sequence $((X_k, d_k), (a^k_n))_{k \in \NN}$ of elements of $C$, let $(\bar X, \bar d)$ be the ultraproduct $\left(\prod_{k \in \NN} (X_k, d_k) \right)_D$, and for each $n$ let $\bar a_n$ be the element of $(\bar X, \bar d)$ represented by $(a^k_n)_{k \in \NN}$. Then for any $\varepsilon > 0$ and $F : \NN \to \NN$, there is an $n$ such that $\bar d (\bar a_i, \bar a_j) < \varepsilon$  for every $i, j \in [n, F(n)]$.
\item For any sequence $((X_k, d_k), (a^k_n))_{k \in \NN}$ of elements of $C$, the sequence $(\bar a_n)$ is Cauchy.
\item For any sequence $((X_k, d_k), (a^k_n))_{k \in \NN}$ of elements of $C$, the sequence $(\bar a_n)$ converges in $(\bar X, \bar d)$.
\end{enumerate}
\end{theorem}

\begin{proof}
As noted in the introduction, the conclusions of (2) and (3) are equivalent in any metric space, and they are clearly equivalent to the conclusion of (4) given the fact that the ultraproduct is complete.

To prove that (1) implies (2), fix  a rational $\varepsilon > 0$ and $F : \NN \to \NN$, and a sequence $((a^k_n), (X_k, d_k)_{k \in \NN})$ of elements of $C$. By (1), there is a $b$ with the property that for every $k$, there is an $n \leq b$ such that $d_k(a^k_i, a^k_j) < \varepsilon / 2$ for every $i, j \in [n, F(n)]$. Thus, for each $k$, the statement $\exists {n \leq b} \; \forall {i, j \in [n, F(n)]} \; d_k(a^k_i, a^k_j) \leq \varepsilon / 2$ is true in $(X_k, d_k)$. Since the existential quantifier can be replaced by finite disjunction and the universal quantifier can be replaced by a finite conjunction, this is equivalent to a quantifier-free positively bounded formula. Hence $\exists {n \leq b} \; \forall {i, j \in [n, F(n)]} \; \bar d(\bar a_i, \bar a_j) \leq \varepsilon / 2$ is true of $(\bar X, \bar d)$, as required.

Conversely, suppose the conclusion of (1) fails for $F$ and $\varepsilon$. For each $k$ in $\NN$, choose a counterexample to the claim for $b = k$, that is, a pair $((X_k, d_k), (a^k_n)_{n \in \NN})$ such that for every $n \leq k$ there are $i, j \in [n, F(n)]$ such that $d_k(a^k_i, a^k_j) \geq \varepsilon$. Then for each $n$ the statement $\exists {i, j \in [n, F(n)]} \; d_k(a^k_i, a^k_j) \geq \varepsilon$ is true for all but finitely many $k$, which implies that $\exists {i, j \in [n, F(n)]} \; \bar d(\bar a_i, \bar a_j) \geq \varepsilon$ is true of $(\bar X, \bar d)$.
\end{proof}

Tao \cite{tao:blog:12} makes use of the equivalence between (2) and (3). In all the applications below, we will use only the implication from (4) to (1), in situations where the metric spaces bear additional structure that is preserved under the formation of ultraproducts.

\section{Applications}
\label{examples:section}

Let $T$ be any nonexpansive operator on a Hilbert space, $\mdl H$, let $f$ be any element of $\mdl H$, and for each $N \geq 1$ let $A_N f$ denote the ergodic average $\frac{1}{n} \sum_{n <N} T^n f$. Riesz's generalization of von Neumann's mean ergodic theorem states that the sequence $(A_N f)$ of averages converges in the Hilbert space norm. The following generalization is due to Lorch \cite{lorch:39}, but also a consequence of results of Riesz~\cite{riesz:38}, Yosida~\cite{yosida:38}, and Kakutani~\cite{kakutani:38} from around the same time (see \cite[p.~73]{krengel:85}). A linear operator $T$ on a Banach space $\mdl B$ is \emph{power bounded} if there is an $M$ such that $\| T^n \| \leq M$ for every $n$.

\begin{theorem}
\label{strong:met}
If $T$ is any power-bounded linear operator on a reflexive Banach space $\mdl B$, and $f$ is any element of $\mdl B$, then the sequence $(A_N f)_{N \in \NN}$ converges. 
\end{theorem}

As noted in Section~\ref{introduction:section}, even in the original von Neumann setting there is no uniform bound on the rate of convergence. Indeed, Fonf, Lin, and Wojtaszczyk \cite{fonf:et:al:01} show that if $\mdl B$ is a Banach space with basis, then there is a uniform bound on the rate of convergence in Theorem~\ref{strong:met} if and only if $\mdl B$ is finite dimensional. Moreover, in the general case, a rate of convergence is not necessarily computable from the given data \cite{avigad:simic:06,vyugin:97}; see also the discussion in \cite[Section 5]{avigad:et:al:10}. However, we can obtain a strong uniformity if we shift attention to metastability. 

\begin{theorem}
\label{uniform:met}
Let $C$ be any class of Banach spaces with the property that the ultraproduct of any countable collection of elements of $C$ is a reflexive Banach space. For every $\rho > 0$, $M$, and function $F : \NN \to \NN$, there is $K$ such that the following holds: given any Banach space $\mdl B$ in $C$, any linear operator on $\mdl B$ satisfying $\| T^n \| \leq M$ for every $n$, any $f \in \mdl B$, and any $\varepsilon > 0$, if $\| f \| / \varepsilon \leq \rho$, then there is an $n \leq K$ such that $\| A_i f - A_j f \| < \varepsilon$ for every $i, j \in [n, F(n)]$.
\end{theorem}

\begin{proof}
Scaling, we can restrict attention to elements $f$ such that $\| f \| \leq 1$. Fix $\rho > 0$ and $M$, and set $\varepsilon = 1 / \rho$. We apply Theorem~\ref{main:theorem} to the class of pairs $( (\mdl B, T, f), (A_n f)_{n \in \NN} )$, where $\mdl B$ is in $C$, $T$ is a linear operator satisfying $\| T^n \| \leq M$ for every $n$, and $\| f \| \leq 1$.

Let $( (\mdl B_k, T_k, f_k), (A_n f_k))_{k \in \NN}$ be any sequence of elements of that class. The fact that $\| T_k \| \leq M$ for every $k$ guarantees that the family $(T_k)$ satisfies the uniform boundedness and uniform continuity conditions. Let $\mdl B = \left( \prod_k \mdl B_k \right)_D$ be the Banach space ultraproduct, and set $T = \left(\prod_k T_k\right)_D$ and $f = \left(\prod_k f_k\right)_D$. By hypothesis, $\mdl B$ is reflexive, and so Theorem~\ref{strong:met} implies that $(A_n f)$ converges in $\mdl B$. By Theorem~\ref{main:theorem}, this implies that there is a uniform bound on the rate of metastability for the sequences $(A_n f)$ occurring in $C$.
\end{proof}

The class $C$ of \emph{all} reflexive Banach spaces does not satisfy the hypothesis of Theorem~\ref{uniform:met}, which is to say, an ultraproduct of reflexive Banach spaces need not be reflexive. However, there are interesting classes $C$ to which the theorem applies. For example, every uniformly convex Banach space is reflexive, and if one fixes a modulus of uniform convexity, the class of uniformly convex spaces with that modulus is closed under ultraproducts. Thus, Theorem~\ref{uniform:met} guarantees the existence of a uniform bound on the rate of metastability that depends only on $\rho$, $M$, $F$, and the modulus of uniform convexity. Avigad and Rute \cite{avigad:rute:unp} show that any such bound for uniformly convex spaces has to depend on the modulus of uniform convexity, and that there is single separable, reflexive, strictly convex Banach space for which the conclusion of Theorem~\ref{uniform:met} fails.

In the case of a linear operator on a uniformly convex Banach space that is either nonexpansive or power-bounded from above and below, Avigad and Rute \cite{avigad:rute:unp} provide a variational inequality which implies an explicit uniform bound on the number of $\varepsilon$-fluctuations of the sequence $(A_n f)$, in terms of $\rho$ and the modulus of uniform convexity. We do not know the extent to which this stronger uniformity extends. (Safarik and Kohlenbach \cite{safarik:kohlenbach:unp} provide some general conditions that guarantee that it is possible to compute a bound on the number of $\varepsilon$-fluctuations.)

For another example of a class $C$ to which Theorem~\ref{uniform:met} applies, say that a Banach space $\mdl B$ is \emph{$J\mbox{-}(n, \varepsilon)$ convex} if for every $x_1, \ldots, x_n$ in the unit ball of $\mdl B$ there is a $j$, $1 \leq j \leq n$, such that
\[
 \big\| \sum_{i < j} x_i - \sum_{i \geq j} x_i \big\| \leq n (1 - \varepsilon).
\]
A space is \emph{$J$-convex} if and only if it is $J\mbox{-}(n, \varepsilon)$ convex for some $n \geq 2$ and $\varepsilon > 0$. Pisier \cite{pisier:11} shows that a Banach space is $J$-convex if and only if it is super-reflexive, so, in particular, every $J$-convex space is reflexive. Moreover, it is immediate from the form of the definition that, for fixed $n \geq 2$ and $\varepsilon > 0$, the class of $J\mbox{-}(n, \varepsilon)$ convex Banach spaces is closed under ultraproducts. Thus, Theorem~\ref{uniform:met} once again guarantees the existence of a uniform bound on the rate of metastability that depends only on $\rho$, $M$, $F$, $n$, and $\varepsilon$. Note that for $n = 2$, a space is $J\mbox{-}(n, \varepsilon)$ convex for some $\varepsilon > 0$ if and only if it is \emph{uniformly non-square}, a weakening of strict convexity due to James \cite{james:64}.

The list of classes of structures to which Theorem~\ref{uniform:met} applies can easily be extended. For example, we can obtain many classes of spaces that satisfy the hypothesis of that theorem by simply fixing bounds on appropriate parameters in the various
characterizations of superstability given by Pisier in Chapter 3 of \cite{pisier:11}. Other examples of classes of reflexive spaces that are closed under formation of ultraproducts can be found in \cite{levy:raynaud:84,raynaud:02,poitevin:raynaud:08}.

We now consider two additional examples, with respect to which notions of metastability have been considered in the past. For the first example, we consider extensions of the mean ergodic theorem to ``diagonal averages.'' Furstenberg's celebrated ergodic-theoretic proof of Szemer\'edi's theorem involves averages of the form 
\[
\frac{1}{n} \sum_{i < n} f_1(T_1^{-i} x) \cdots f_j(T_j^{-i} x)
\]
where $T_1,\ldots,T_j$ are commuting measure-preserving transformations of a finite measure space $(X, \mdl X, \mu)$. Settling a longstanding open problem, Tao \cite{tao:08} showed that such sequences always converge in the $L^2(X)$ norm. This result was recently generalized by Walsh \cite{walsh:12}, as follows:

\begin{theorem}
\label{walsh:theorem}
Let $(X, \mdl X, \mu)$ be a finite measure space with a measure-preserving action of a nilpotent group $G$. Let $T_1, \ldots, T_l$ be elements of $G$, and let 
\[
(p_{i, j})_{i = 1, \ldots, l; j = 1, \ldots, d}
\]
be a sequence of integer-valued polynomials on $\ZZ$. Then for any $f_1, \ldots, f_d \in L^\infty(X, \mdl X, \mu)$, the sequence of averages
\[
 \frac{1}{N} \sum_{n = 1}^N \prod_{j = 1}^d \left(T_1^{p_{1,j}(n)} \cdots T_l^{p_{l, j}(n)} \right) f_j
\]
converges in the $L^2(X)$ norm.

\end{theorem}

When the relevant data $\vec T, \vec p$ are clear, it will be convenient to write $A_N(\vec f)$ for these averages. Once again, a compactness argument yields the following uniformity:

\begin{theorem}
\label{uniform:walsh:theorem}
For every $r$, $l$, $d$, $s$, $\rho > 0$, and function $F : \NN \to \NN$, there is a $K$ such that the following holds: given a nilpotent group $G$ of nilpotence class at most $r$, elements $T_1, \ldots, T_l$ in $G$, a sequence $(p_{i, j})_{i = 1, \ldots, l; j = 1, \ldots, d}$
of integer-valued polynomials on $\ZZ$ of degree at most $s$, a probability space $(X, \mdl X, \mu)$, a measure-preserving action of $G$ on $(X, \mdl X, \mu)$, and any sequence of elements $f_1, \ldots, f_d \in L^\infty(X, \mdl X, \mu)$, if $\| f_i \|_\infty / \varepsilon \leq \rho$ for each $i$, then there is an $n \leq K$ such that $\| A_i(\vec f) - A_j(\vec f) \| < \varepsilon$ for every $i, j \in [n, F(n)]$.
\end{theorem}

As above, we can restrict attention to the case where $\| f_i \|_\infty \leq 1$ in the statement of the theorem, and without loss of generality we can assume that $G$ is generated by $T_1, \ldots, T_l$. An ultraproduct construction due to Loeb~\cite{loeb:75}, analogous to the constructions described in Section~\ref{introduction:section}, can be used to amalgamate a sequence of measure spaces $(X_k, \mdl X_k, \mu_k)$ to a measure space $(X, \mdl X, \mu)$, and since first-order properties of discrete structures are preserved under ultraproducts, the ultraproduct of a sequence $(G_k)$ of groups of nilpotence class at most $r$ is again a group of nilpotence class at most $r$. A measure-preserving action of each $G_k$ on $(X_k, \mdl X_k, \mu_k)$ gives rise to a measure-preserving action of $G$ on $(X, \mdl X, \mu)$, and the product of the spaces $L^2(X_k, \mdl X_k, \mu_k)$ embeds isometrically into the space $L^2(X, \mdl X, \mu)$ (see, for example, \cite[Section 5]{heinrich:80}).

There is a catch, though: the ultraproduct of a sequence of polynomials $p_k$ with coefficients in $\ZZ$ need not be a polynomial, since the coefficients can ``go off to infinity.'' One could rule that out by assuming that there is a uniform bound on those coefficients, in which case the value $K$ in the statement of the theorem would depend on that bound as well. As it turns out, however, in this particular case there is a trick that eliminates the dependence on this parameter. Call a sequence $(g_n)$ of elements of the form $g_n = T_1^{p_{1}(n)} \cdots T_l^{p_{l}(n)}$ a \emph{polynomial sequence}.
\begin{lemma}
\label{nilpotent:trick:lemma}
 Let $G$ be a nilpotent group, and let $(g_n)$ be a polynomial sequence of elements of $G$. Then there are a nilpotent extension $\eta : \hat G \to G$ and elements $\tau$ and $c$ of $\hat G$ such that for every $n$, $g_n = \eta(\tau^n c)$. Moreover, there is a bound on the nilpotence class of $\hat G$ that depends only on bounds on the nilpotence class of $G$, the number $l$ of polynomials, and a bound on their degrees.
\end{lemma}
Via $\eta$, the action of $G$ on $X$ lifts to an action of $\hat G$ on $X$, whereby the action of $g_n$ lifts to the action of $\tau^n c$. Applying the lemma $d$ times, we can thus assume that each polynomial sequence $g_{i,n} = T_1^{p_{i,1}(n)} \cdots T_l^{p_{i,l}(n)}$ appearing in the statement of Walsh's theorem is of the form $\tau_i^n c_i$ for some $\tau_i$ and $c_i$ in $G$, at the expense of increasing the nilpotence rank of $G$.

Lemma~\ref{nilpotent:trick:lemma} is a special case of a construction carried out by Leibman \cite{leibman:05} in the more general setting of an action of Lie group, with both continuous and discrete elements. We are grateful to Terence Tao for bringing this lemma to our attention, and pointing out that it can be used to obtain a stronger uniformity in the statement of Theorem~\ref{uniform:walsh:theorem}. As Leibman points out, an instance of this trick was used by Furstenberg \cite[page 31]{furstenberg:81}. Leibman's construction can be divided into two parts: Proposition 3.14 of \cite{leibman:05} shows how to define a nilpotent extension $\eta : \tilde G \to G$, a unipotent automorphism $\tau$ of $\tilde G$, and an element $c$ of $\tilde G$, such that for every $n$, $g(n) = \eta(\tau^n(c))$; and Proposition 3.9 of \cite{leibman:05} shows that the extension $\hat G$ of $\tilde G$ by $\tau$ is again a nilpotent group. Here, saying that $\tau$ is a unipotent automorphism means that the mapping $\xi(a) = \tau(a) a^{-1}$ has the property that $\xi^q$ is the identity for sufficiently large $q$. Leibman's proof of Proposition 3.9 gives an explicit bound on how large $q$ has to be and the nilpotence class of $\tilde G$; and Proposition 1 of Gruenberg \cite{gruenberg:59} then provides the requisite bound on the nilpotence class of $\hat G$.

With this lemma in hand, we can prove Theorem~\ref{uniform:walsh:theorem}. 

\begin{proof}
As above, we can restrict attention to the case where $\| f_i \|_\infty \leq 1$ in the statement of the theorem. Using Lemma~\ref{nilpotent:trick:lemma}, we can moreover assume $d = 2l$, $s = 1$, and for every $i$, $p_{i,2i}(n) = n^i$, $p_{i, 2i + 1}(n) = 1$, and $p_{i,j} = 0$ for all other $j$, so that the $i$th polynomial sequence is given by $g_{i,n} = T_{2i}^n T_{2i+1}$.  

Once again, we fix $r$, $l$, and $\rho$, and use Theorem~\ref{main:theorem} with $\varepsilon = 1 / \rho$. Suppose we are given, for each $k$, a probability space $(X_k, \mdl X_k, \mu_k)$, a group $G_k$ of nilpotence class at most $r$, elements $T_{1,k}, \ldots, T_{l,k}$, and elements $f_{1, k}, \ldots, f_{d, k}$ with infinity norm at most 1. Let $(X, \mdl X, \mu)$ be the result of applying the Loeb construction to the sequence of spaces $(X_k, \mdl X_k, \mu_k)$, let $G$ be the ultraproduct of the sequence $(G_k)$ with respect to $D$. For each $i$ let $T_i = \left( \prod_k T_{i,k} \right)_D$, and for each $j$ let $f_j = \left( \prod_k f_{j,k} \right)_D$. Then $G$ has nilpotence class at most $r$, and each $T_i$ is measure-preserving transformation of $X$, so Theorem~\ref{walsh:theorem} implies convergence of the sequence $(A_n (\vec f))$. By Theorem~\ref{main:theorem}, this implies a uniform bound on the rate of metastability.
\end{proof}

Tao \cite{tao:blog:12} shows that one can alternatively formulate Walsh's theorem in algebraic terms, which allows one to avoid the reference to the Loeb construction in the proof of Theorem~\ref{uniform:walsh:theorem}. In fact, both Walsh's original proof \cite{walsh:12} and Tao's later proof of Walsh's result \cite{tao:blog:12} establish Theorem~\ref{uniform:walsh:theorem} directly. Tao's proof of his prior result \cite{tao:08} also established the corresponding uniformity, but there are now other proofs of that theorem that do not \cite{austin:10,host:09,towsner:09b}. Tao \cite{tao:blog:12} emphasizes that Theorem~\ref{uniform:walsh:theorem} is stronger than Theorem~\ref{walsh:theorem}; the observation here is that they are essentially the same, modulo compactness and Lemma~\ref{nilpotent:trick:lemma}.

We consider a final example, this time from nonlinear ergodic theory. Fix a Hilbert space $\mdl H$. Let $C$ be a closed, convex subset of $\mdl H$, and let $T$ be a nonexpansive map from $C$ to $C$. Let $(\lambda_n)$ be a sequence of elements of $[0,1]$, and let $f$ and $u$ be any elements of $C$. The \emph{Halpern iteration} corresponding to $T$, $(\lambda_n)$, $f$, and $u$ is the sequence given by
\[
 f_0 = f, \quad f_{n+1} = \lambda_{n+1} u + (1 - \lambda_{n+1}) T f_n.
\]
If $T$ is linear, $u = f$, and $\lambda_n = 1 / (n + 1)$, then $(f_n)$ is the familiar sequence $(A_n f)$ of ergodic averages. Wittmann \cite{wittmann:92} showed that, assuming the set of fixed points of $T$ is nonempty, the following conditions on the sequence $(\lambda_n)$ suffice to ensure that the sequence $f_n$ of Halpern iterates converges to the projection onto the space of fixed points:
\begin{itemize}
 \item $\lim_{n \to \infty} \lambda_n = 0$
 \item $\sum_{n = 1}^\infty \| \lambda_{n+1} - \lambda_n \|$ converges
 \item $\sum_{n = 1}^\infty \lambda_n = \infty$
\end{itemize}
In particular, these are satisfied when $\lambda_n = 1 / (n + 1)$.

The linear structure of $\mdl H$ only comes into play in the assumption that $C$ is convex. Seajung \cite{saejung:10} has generalized Wittmann's result to CAT(0) spaces. These are metric spaces with an abstract notion of ``linear combination,'' that is, metric spaces equipped with a function $W(x, y, \lambda)$ which, intuitively, plays the role of $(1 - \lambda) x + \lambda y$. The specific axioms that $W$ is assumed to satisfy can be found in \cite{bridson:haefliger:99,kohlenbach:leustean:12,saejung:10}; we only need the fact, established in \cite[pages 77--78]{bridson:haefliger:99}, that the ultraproduct of CAT(0) spaces is again a CAT(0) space. Saejung's theorem states the following:

\begin{theorem}
\label{saejung:thm}
Let $C$ be a closed convex subspace of a complete CAT(0) space, and let $T : C \to C$ be a nonexpansive map such that the set of fixed points of $T$ is nonempty. Suppose $(\lambda_n)$ satisfies the three conditions above. Then for any $u, f$ in $C$, the sequence of Halpern iterates $(f_n)$ converges to the projection of $u$ onto the set of fixed points of $T$. 
\end{theorem}

If $g$ is a fixed point of $T$ and $b = \max( \| f - g \|, \| u - g \| )$, then it is not hard to show that one can restrict attention to $C \cap B(g, b)$ in the statement of Theorem~\ref{saejung:thm}. In other words, there is no loss of generality in assuming that $C$ has a bounded diameter. Kohlenbach and Leu\c{s}tean \cite{kohlenbach:leustean:12} have shown that in that case there is a uniform bound on the rate of metastability, given by a primitive recursive functional, which depends on the diameter of $C$. If one is only interested in uniformity and not the particular rate, the following provides a quick proof:

\begin{theorem}
\label{uniform:saejung:thm}
Fix $(\lambda_n)$ satisfying (1--3) above. For every $\varepsilon > 0$, $M$, and function $F : \NN \to \NN$, there is a $K$ such that the following holds: given a CAT(0) space $(X, d, W)$, a closed convex subset $C$ of $X$ with diameter at most $M$, a nonexpansive map $T : C \to C$ with a fixed point in $C$, and $f, u$ in $C$, if $(f_n)$ denotes the sequence of Halpern iterates, then there is an $n \leq K$ such that $d(f_i, f_j) < \varepsilon$ for every $i, j$ in $[n, F(n)]$.
\end{theorem}

\begin{proof}
  Once again, we apply Theorem~\ref{main:theorem}. We have already noted that the ultraproduct of CAT(0) spaces is again a CAT(0) space. The uniform bound on the diameter of each of the sets $C$ is also a bound on the diameter of their product. The fact that convexity is preserved is immediate, and it is not hard to show that an ultraproduct of closed sets is again closed (see, for example, \cite[Proposition 5.3]{ben:yaacov:et:al:08}).
\end{proof}

Theorem~\ref{uniform:saejung:thm} can also be seen as a consequence of Corollary 4.26 in Gerhardy and Kohlenbach \cite{gerhardy:kohlenbach:08}, modulo verification of the fact that Saejung's theorem can be derived in the formal axiomatic system mentioned there. That corollary ensures, moreover, that the bound is computable from the parameters.

Under the assumption that $C$ is a bounded, closed, convex subset of a CAT(0) space, Kirk \cite[Theorem 18]{kirk:03} shows that a nonexpansive map from $C$ to $C$ necessarily has a fixed point. Thus in the statement of Theorem~\ref{uniform:saejung:thm} that hypothesis could be dropped. Gerhardy and Kohlenbach note that, in more general situations, one can weaken the hypothesis that $T$ has a fixed point in $C$ to the hypothesis that $T$ has an $\varepsilon$-fixed point in $C$ for every $\varepsilon >0$. This is easy to see from the ultraproduct argument as well, since an ultralimit of $\varepsilon$-fixed points for a sequence $\varepsilon$ decreasing to $0$ is an actual fixed point. This fact is commonly used in applications of ultraproducts to fixed-point theory; see, for example, Aksoy and Khamsi \cite{aksoy:khamsi:90}.


\end{document}